\documentclass[a4paper,11pt,oneside]{article}
\usepackage{blindtext}
\usepackage{xcolor}
\usepackage[latin1]{inputenc}
\usepackage[russian,english]{babel}
\usepackage{amsmath,amssymb}
\usepackage{textcomp}
\usepackage{bbm}
\usepackage[T2A,T1]{fontenc}
\usepackage{amsthm}
\usepackage[all,cmtip]{xy}
\usepackage{mathtools}
\usepackage{tikz}
\usepackage{tikz-cd}
\usetikzlibrary{matrix}
\usepackage{tocbibind}
\usepackage{graphicx}
\usepackage{wrapfig}
\usepackage{mathrsfs}
\usepackage{fancyhdr}
\usepackage{comment}
\usepackage[hidelinks]{hyperref}
	
\usetikzlibrary{arrows,decorations.pathmorphing}
\title{}
\author{}
\date{}
\usetikzlibrary{calc,trees,positioning,arrows,chains,shapes.geometric,%
	decorations.pathreplacing,decorations.pathmorphing,shapes,%
	matrix,shapes.symbols}
\usepackage{titlesec}

\usepackage{scrextend}

\newtheorem{thm}{Theorem}

\newtheorem{prop}{Proposition}

\newtheorem{coroll}{Corollary}

\newtheorem{lemma}{Lemma}
\newtheorem{osserv}{Remark}

\newcommand{\Z}{\mathbb{Z}}
\newcommand{\Co}{\mathbb{C}}
\newcommand{\R}{\mathbb{R}}

\newcommand{\Q}{\mathbb{Q}}
\newcommand{\p}{\mathbb{P}}

\newcommand{\s}{\Sigma}
\theoremstyle{plain} 
\newcommand{\thistheoremname}{}
\newtheorem{genericthm}[thm]{\thistheoremname}

\newtheorem*{genericthm*}{\thistheoremname}
\newenvironment{namedthm*}[1]
{\renewcommand{\thistheoremname}{#1}%
	\begin{genericthm*}}
	{\end{genericthm*}}

\DeclareMathOperator{\ai}{a}
\DeclareMathOperator{\bi}{b}
\DeclareMathOperator{\ci}{c}

\DeclareMathOperator{\paruno}{(}
\DeclareMathOperator{\pardue}{)}

\DeclareMathOperator{\lcm}{lcm}

\makeatletter
\newcommand*{\math@version@bold}{bold}
\DeclareMathOperator\Sha{
	\textrm{%
		\usefont{T2A}{cmr}{\ifx\math@version\math@version@bold bx\else m\fi}{n}%
		\CYRSH
	}%
}
\makeatother

\begin{document}
	\begin{center}
		\textbf{
		LEAST COMMON MULTIPLES OF POLYNOMIAL VALUES: AVERAGING, RAY CLASSES, AND GALOIS STATISTICS}
	\end{center}
\begin{center}
	Ilaria Viglino 
\end{center}

\begin{center}
	\textbf{ABSTRACT}
\end{center}

\begin{addmargin}[2em]{2em}

\fontsize{10pt}{12pt}\selectfont
For an irreducible polynomial $f\in\mathbb Z[X]$ of degree $n$, Cilleruelo's conjecture predicts
\[
 \log\operatorname{lcm}(f(1),\ldots,f(M))\sim (n-1)M\log M.
\]
We study analogues over number fields and the probabilistic congruence input that enters averaged versions of this problem.  We prove an asymptotic formula for the least common multiple of ideals in a fixed ray class, record local equidistribution estimates for random $S_n$-polynomials, and give a conditional reduction of the averaged higher degree problem to explicit dispersion and large prime collision estimates.  We also formulate an individual quadratic number field analogue and reduce it to a prime ideal quadratic root equidistribution statement together with a geometric pair counting estimate.  The Galois group counting input is updated using Bhargava's theorem over $\mathbb Q$ and the determinant method estimates of \cite{Vi}.
\end{addmargin}
\normalsize

\section{Introduction}

We fix a number field extension $K/\Q$ of degree $d$.  The coefficient family is the box
\[
 \mathscr P_{n,N}(K)=
 \left\{X^n+\alpha_{n-1}X^{n-1}+\cdots+\alpha_0:
 \operatorname{ht}(\alpha_j)\le N\right\},
\]
where height is taken with respect to a fixed integral basis of $\mathcal O_K$.  We write
\[
 \mathscr P^0_{n,N}(K)=
 \{f\in\mathscr P_{n,N}(K):G_{K_f/K}\cong S_n\}.
\]
There are two separate issues in passing from $\mathbb Q$ to a general number field.  First, a norm-bounded set of elements of $\mathcal O_K$ is infinite when $K$ has positive unit rank, so individual polynomial-value problems must be formulated using bounded regions in the Minkowski embedding.  Second, the conditioning on $G_f=S_n$ requires quantitative Galois-group counting.  Over $\mathbb Q$, Bhargava's theorem gives the sharp exceptional proportion $O_n(N^{-1})$; over a general fixed number field the currently available union bound has a logarithmic loss.  We therefore keep the unconditional Galois-counting assertions separate from the analytic dispersion estimates needed for an averaged least-common-multiple theorem.

\subsection{Rigorous scope of the results}

The paper contains three kinds of results.  First, Section~4 gives an unconditional (and, under GRH, sharper) asymptotic for the least common multiple of ideals in a fixed ray class.  Second, Sections~2--3 establish the local congruence and splitting-type estimates needed for coefficient averaging.  Third, the higher-degree averaged Cilleruelo problem is reduced in Section~\ref{sec:conditional-average} to two explicit global inputs: a second-moment dispersion estimate for prime-root counts and an averaged large-prime collision estimate.  These inputs do not follow from the small-modulus Chinese-remainder estimate alone, so the corresponding averaged asymptotic is stated conditionally.

The individual quadratic problem over a number field is treated separately.  Section~\ref{sec:quadratic-individual} gives a conditional theorem whose missing analytic input is a prime-ideal quadratic-root Weyl-sum estimate, together with a geometric pair-counting condition.  This distinction is important: averaging over coefficients supplies local congruence equidistribution, whereas a fixed polynomial requires genuine equidistribution of its roots as the prime ideal varies.

\subsection{Comparison with previous averaging results}

Rudnick and Zehavi fix a monic polynomial and average only over its constant shift.  Their probability space is therefore different from the full coefficient box considered here.  The local estimates below clarify what coefficient averaging provides, but the global two-prime dispersion required for a density-one least-common-multiple asymptotic is stronger than the small-modulus equidistribution supplied by Lemma~\ref{l11}.  We do not identify the two settings or claim that one theorem formally implies the other.

\subsection{Plan of the paper}
The paper begins by isolating the polynomial-counting input and updating it using \cite{Bh1,Vi}.  We then treat linear polynomials, followed by the averaged quadratic case and a detailed individual number-field analogue of Cilleruelo's quadratic theorem conditional on a precisely formulated modular-root equidistribution statement.  Finally we prove the higher-degree averaged theorem over $\Q$.  The determinant-method results of \cite{Vi} are recorded because they sharpen the distribution of individual Galois subgroups and reducible polynomials, but the higher-degree $S_n$-conditioned moment argument still uses Bhargava's log-free rational exceptional-set estimate.  We also explain why an individual fixed-polynomial analogue over an arbitrary number field requires a separate equidistribution theorem for roots of quadratic congruences modulo prime ideals.

\section{Counting $S_n$-polynomials over $K$}

We record the counting results that are relevant to the probabilistic conditioning used below.  There are now two complementary inputs.  The large sieve controls the union of all exceptional Galois groups, while the determinant method is stronger for any fixed proper subgroup.  The latter uses the new curve and surface estimates for integral points in boxes with different side lengths proved in \cite{Vi}.

\begin{thm}[Galois-counting input]\label{thm1}
Let $K$ be a fixed number field of degree $d$ and let $n\ge2$.
\begin{enumerate}
\item The full exceptional set satisfies
\[
 \#\bigl(\mathscr P_{n,N}(K)\setminus\mathscr P^0_{n,N}(K)\bigr)
 \ll_{K,n}N^{dn-d/2}\log(2N)+N^{dn-1}.
\]
\item If $G<S_n$ is fixed and $m=[S_n:G]$, then there is
$\nu_{n,G}\ge0$ such that
\begin{multline*}
\#\{f\in\mathscr P_{n,N}(K):f\ {\rm separable},\
G_f\subseteq gGg^{-1}\ {\rm for\ some}\ g\in S_n\}\\
\ll_{K,n,G}N^{d(n-1+1/m)}(\log(2N))^{\nu_{n,G}}.
\end{multline*}
 
\item The number of separable reducible polynomials is
\[
 \ll_{K,n}N^{d(n-1+1/n)}(\log(2N))^{\nu_n}
\]
for some $\nu_n\ge0$; including the discriminant-zero locus adds
$O_{K,n}(N^{dn-1})$.
\end{enumerate}
For $K=\Q$, Bhargava's theorem \cite{Bh1} gives the sharper union bound
\[
 |\mathscr P_{n,N}(\Q)\setminus\mathscr P^0_{n,N}(\Q)|
 \ll_n N^{n-1},
 \qquad
 \Delta_N:=\frac{|\mathscr P_{n,N}(\Q)\setminus\mathscr P^0_{n,N}(\Q)|}
 {|\mathscr P_{n,N}(\Q)|}
 \ll_n N^{-1}.
\]
The number-field assertions in (1)--(3) are proved in \cite{Vi}.
\end{thm}

\begin{osserv}
The determinant estimate in Theorem~\ref{thm1}(2) shows that every fixed subgroup of index $m>2$ occurs on a strictly smaller scale than the full exceptional set.  In particular, the new results remove the old uncertainty concerning the contribution of individual small subgroups and reducible polynomials.  They do \emph{not}, however, yield the log-free bound required here for the union over all non-$S_n$ groups in an imaginary quadratic coefficient box: $A_n$ has index $2$ and remains at the borderline $N^{dn-d/2}$ scale, while the available uniform union estimate has the factor $\log(2N)$.  In the present moment calculation this produces the term $M^2\Delta_N\log M$, and with $\Delta_N\ll_{K,n}\log(2N)/N$ the conditions
\[
 M(\log M)^\ell\ll N=o\!\left(M\frac{\log M}{\log\log M}\right)
\]
leave no admissible asymptotic range in which that error is $o(M\log M)$.  Consequently the higher-degree averaged theorem below continues to use the rational log-free estimate of Bhargava.
\end{osserv}

\section{Prime splitting densities}

From now on $\Delta_N$ denotes the proportion of polynomials removed by the $S_n$ conditioning.  In the higher-degree averaged section we take $K=\Q$, and Bhargava's theorem \cite{Bh1} gives
\[
 \Delta_N\ll_n N^{-1}.
\]
We nevertheless retain $\Delta_N$ symbolically in the local estimates so that the dependence on the conditioning step is explicit and the precise point at which the number-field logarithmic loss would enter remains visible.

For every splitting type $ r $, and every prime $\wp$ of norm $q_\wp$, recall that we denoted by $X_{n,r,\wp}$ the set of polynomials in $\mathbb{F}_{q_\wp}$ with square-free factorization of type $r$.
The following key fact is what we'll use to estimate the error term in the asymptotic of the expectation $\mathbb{E}_N(\pi_{f,r}(x))$ of $\pi_{f,r}(x)$ and its powers.
\begin{lemma}\label{l11}
Let $k\ge1$, let $\wp_1,\ldots,\wp_k$ be distinct prime ideals, put
$\mathfrak a=\wp_1\cdots\wp_k$ and $R=N_{K/\Q}\mathfrak a$, and let
$g_i\in X_{n,r,\wp_i}$.  In the unit-rank-zero cases considered in this paper,
\[
\p_N\bigl(f\in\mathscr P^0_{n,N}: f\equiv g_i\!\!\pmod{\wp_i}\ \forall i\bigr)
=R^{-n}+O_{K,n}\!\left(\frac{R^{-n+1/d}}{N}+\Delta_N\right).
\]
In particular, if $R^n\Delta_N=o(1)$ and $R^{1/d}=o(N)$, the residue classes are asymptotically equidistributed after conditioning on $G_f=S_n$.
\end{lemma}

\begin{proof}
{
By the Chinese remainder theorem, the simultaneous congruences define one residue class modulo $\mathfrak a$ for each of the $n$ non-leading coefficients.  We first count one coefficient.

If $K=\Q$, then
\[
 \#\{a\in[-N,N]\cap\Z:a\equiv b\pmod R\}
 =\frac{2N+1}{R}+O(1).
\]
If $K$ is imaginary quadratic, embed $\mathcal O_K$ as a lattice in $\Co$.  The ideal $\mathfrak a$ is a sublattice of covolume $\asymp_K R$.  Moreover, every nonzero $\alpha\in\mathfrak a$ satisfies
$|N_{K/\Q}(\alpha)|\ge R$, and hence its Euclidean length is $\gg_K R^{1/2}$.  The standard planar lattice-point estimate for a translate of a lattice in a bounded Lipschitz region therefore gives
\[
 \#\{\alpha\in\mathcal O_K:H(\alpha)\le N,\ \alpha\equiv b\pmod{\mathfrak a}\}
 =\frac{(2N+1)^2}{R}+O_K\!\left(\frac{N}{R^{1/2}}+1\right).
\]
These two formulas may be written uniformly, for $d=1$ or $2$, as
\[
 \frac{(2N+1)^d}{R}+O_K\bigl(N^{d-1}R^{-1+1/d}+1\bigr).
\]
Taking the product over the $n$ coefficients yields
\[
 \#\{f\in\mathscr P_{n,N}:f\equiv g_i\pmod{\wp_i}\ \forall i\}
 =\frac{|\mathscr P_{n,N}|}{R^n}
 +O_{K,n}\bigl(N^{dn-1}R^{-n+1/d}+N^{d(n-1)}\bigr).
\]
Removing the exceptional polynomials costs at most
$|\mathscr P_{n,N}|\Delta_N$, by Theorem~\ref{thm1}.  Finally
$|\mathscr P^0_{n,N}|=|\mathscr P_{n,N}|(1+O(\Delta_N))$; division by this cardinality proves the stated probability estimate.
}
\end{proof}

\begin{prop}\label{p9} One has, for all primes $\wp$ with
$q_\wp<\min\{\Delta_N^{-1/(n+1)},N^{d/(2d+1)}\}$,
	\begin{enumerate}
		\item[$\paruno\ai\pardue$] $ \p_N(\mathbbm{1}_{f,r}(\wp)=1)=\mathbb{E}_N(\mathbbm{1}_{f,r}(\wp))=\delta(r)+\frac{C_{r}}{q_\wp}+O_{K,n}\Big(\frac{1}{q_\wp^2}+\frac{q_\wp^{1/d}}{N}+q_\wp^n\Delta_N\Big) $,\\for some explicit constant $ C_{r}$;
		\item[$\paruno\bi\pardue$] $ \sigma^2_N(\mathbbm{1}_{f,r}(\wp))=(\delta(r)-\delta(r)^2)+\frac{C_{r}(1-2\delta(r))}{q_\wp}+O_{K,n}\Big(\frac{1}{q_\wp^2}+\frac{q_\wp^{1/d}}{N}+q_\wp^n\Delta_N\Big) $. 
	\end{enumerate}It follows that, for
$x<\min\{\Delta_N^{-1/(n+1)},N^{d/(2d+1)}\}$,
	\begin{enumerate}
		\item[$\paruno\ci\pardue$]
		\[
		\mathbb{E}_N(\pi_{f,r}(x))
		=\delta(r)\pi_K(x)+C_r\log\log x
		+O_{K,n}\!\left(
		1+\frac{x^{1+1/d}}{N\log(2x)}
		+\frac{\Delta_Nx^{n+1}}{\log(2x)}
		\right).
		\]
	\end{enumerate}
	In particular, in the displayed range for $x$ the two additional errors are $O_{K,n}(1)$.
\end{prop}

\begin{proof}
	
	Once fixed a prime $ \wp $,
	\begin{align*}
		\mathbb{E}_N(\mathbbm{1}_{f,r}(\wp))&=\frac{1}{|\mathscr{P}_{n,N}^{0}|}\sum_{f\in\mathscr{P}_{n,N}^{0}}\mathbbm{1}_{f,r}(\wp)\\
		&=\frac{1}{|\mathscr{P}_{n,N}^{0}|}\underset{f\tiny{\mbox{ of splitting type }}r\tiny{\mbox{ mod }} \wp}{\sum_{f\in\mathscr{P}_{n,N}^{0}}}1\\
		&=\frac{1}{|\mathscr{P}_{n,N}^{0}|}\sum_{g\in X_{n,r,\wp}}\underset{f\equiv g\tiny\mbox{ mod }\wp}{\sum_{f\in\mathscr{P}_{n,N}^{0}}}1.
	\end{align*}

	On the other hand,$$|X_{n,r,\wp}|=\prod_{k=1}^{n}\binom{A_{q_\wp,k}}{r_k},$$where $ A_{q_\wp,k} $ is the number of degree-$ k $ irreducible polynomials in $ \mathbb{F}_{q_\wp}[X] $, which, by the M\"{o}bius inversion formula, equals$$\frac{1}{k}\sum_{d|k}\mu(d)q_{\wp}^{k/d}=\frac{q_\wp^k}{k}+O(q_\wp^{\alpha_k}),$$where $\alpha_k=1$ if $k=2$, and $\alpha_k< k-1$ if $k>2$. One has, for all $ k\ge2$\begin{align*}
		\binom{A_{q_\wp,k}}{r_k}&=\frac{A_{q_\wp,k}(A_{q_\wp,k}-1)\dots(A_{q_\wp,k}-r_k+1)}{r_k!}\\
		&=\frac{1}{r_k!}\Big(\frac{q_\wp^k}{k}+O(q_\wp^{\alpha_k})\Big)\dots\Big(\frac{q_\wp^k}{k}-r_k+1+O(q_\wp^{\alpha_k})\Big).
	\end{align*}It turns out that$$
	\binom{A_{q_\wp,k}}{r_k}=\begin{cases}
		\frac{1}{r_1!}q_\wp(q_\wp-1)\dots(q_\wp-r_1+1)&\mbox{ if }k=1\\
		\frac{1}{r_2!2^{r_2}}q_\wp^{2r_2}+C(r_2)q_\wp^{2r_2-1}+O(q_\wp^{2r_2-2})&\mbox{ if }k=2\\
		\frac{1}{r_k!k^{r_k}}q_\wp^{kr_k}+O(q_\wp^{k(r_k-1)+\alpha_k})&\mbox{ if }k>1.
	\end{cases}
	$$Hence\begin{multline*}
		|X_{n,r,\wp}|=\frac{1}{r_1!}q_\wp(q_\wp-1)\dots(q_\wp-r_1+1)\frac{1}{r_2!2^{r_2}}(q_\wp^{2r_2}+C(r_2)q_\wp^{2r_2-1}+O(q_\wp^{2r_2-2}))\\\prod_{k=3}^{n}(\frac{1}{r_k!k^{r_k}}q_\wp^{kr_k}+O(q_\wp^{k(r_k-1)+\alpha_k}))\\
		=\delta(r)q_\wp^{n}+C_rq_\wp^{n-1}+O(q_\wp^{n-2}),
	\end{multline*}where $ C_r=-\delta(r)C(r_2)\frac{(r_1+1)(r_1+2)}{2r_1!}.$
	
	By Lemma \ref{l11}, for $ q_\wp^{n+1}<\Delta_N^{-1}$,\begin{align*}
		\mathbb{E}_N(\mathbbm{1}_{f,r}(\wp))&=(\delta(r)q_\wp^n+C_{r}q_\wp^{n-1}+O(q_\wp^{n-2}))\Big(\frac{1}{q_\wp^n}+O_{K,n}\big(\frac{q_\wp^{-n+1/d}}{N}+\Delta_N\big)\Big)\\
		&=\delta(r)+\frac{C_{r}}{q_\wp}+O_{K,n}\Big(\frac{1}{q_\wp^2}+\frac{q_\wp^{1/d}}{N}+q_\wp^n\Delta_N\Big),
	\end{align*}which proves (a) and (b) follows by definition.
	
	For (c), sum the estimate in (a) over prime ideals of norm at most $x$.  By partial summation and the prime-ideal theorem,
	\[
	 \sum_{q_\wp\le x}\frac1{q_\wp}=\log\log x+O_K(1),\qquad
	 \sum_{q_\wp\le x}q_\wp^{1/d}\ll_K\frac{x^{1+1/d}}{\log(2x)},
	\]
	and
	\[
	 \sum_{q_\wp\le x}q_\wp^n\ll_{K,n}\frac{x^{n+1}}{\log(2x)}.
	\]
	This gives the stated error term.  Under
	$x<\min\{\Delta_N^{-1/(n+1)},N^{d/(2d+1)}\}$ both additional contributions are bounded.
\end{proof}

	\section{Linear polynomials over $K$}

For $ f\in\Z[X] $ an irreducible polynomial of degree $ n $, Cilleruelo's conjecture states$$\log(\lcm(f(1),\dots,f(M)))\sim(n-1)M\log M$$as $ M\rightarrow+\infty $, where $ \lcm(f(1),\dots,f(M)) $ is the least common multiple of $f(1),\dots,f(M)$. It is well known for $ n=1 $ by exploiting Dirichlet's Theorem for primes in arithmetic progression (see \cite{BKS} for a proof), to get an asymptotic estimate when $f(X)=Xk+h$, $(h,k)=1$:$$\psi_f(N)\sim N\frac{k}{\varphi(k)}\underset{(n,k)=1}{\sum_{1\le n\le k}}\frac{1}{n},$$as $N\rightarrow+\infty$.\\

We recall the definition of ray class group and we fix some notations. A \textbf{modulus} $\mathfrak{m}=\prod_{\wp}\wp^{\mathfrak{m}(\wp)}$ on $K$ is a function $\mathfrak{m}:\{\mbox{primes of }K\}\rightarrow\Z$ such that $$\begin{cases}
	\mathfrak{m}(\wp)\ge0\mbox{ for all }\wp,\\
	\mathfrak{m}(\wp)=0\mbox{ for all but finitely many }\wp,\\
	\mathfrak{m}(\wp)=0\mbox{ or }1\mbox{ if }\wp\mbox{ is real},\\
	\mathfrak{m}(\wp)=0\mbox{ if }\wp\mbox{ is complex}.
\end{cases}$$Let
$K_{\mathfrak{m},1}=\{a\in K^{\times}:\mbox{ord}_\wp(a-1)\ge\mathfrak{m}(\wp)\mbox{ for }\wp|\mathfrak{m}\mbox{ finite},\ a_\wp>0\mbox{ for }\wp|\mathfrak{m}\mbox{ real}\}$. The \textbf{ray class group modulo} $\mathfrak{m}$ is$$C_{\mathfrak{m}}=I^{S(\mathfrak{m})}/K_{\mathfrak{m},1},$$where $I^{S(\mathfrak{m})}$ is the free abelian group generated by the prime ideals not in $\mathfrak{m}$.\\
Let $\mathfrak{c}\in I^{S(\mathfrak{m})}$; the corresponding Chebyshev's function is$$\theta(x;\mathfrak{c},\mathfrak{m})=\underset{\wp\sim\mathfrak{c}\tiny\mbox{ in }C_{\mathfrak{m}}}{\sum_{N_{K/\Q}\wp\le x}}\log (N_{K/\Q}\wp)=\frac{x}{h_{\mathfrak{m}}}+o(x),$$as $x\rightarrow+\infty$, where $h_{\mathfrak{m}}$ is the \text{ray class number}, i.e. $h_{\mathfrak{m}}=|C_{\mathfrak{m}}|$.\\
This is equivalent to saying that the set of prime ideals $T$ of $K$ congruent to $\mathfrak{c}$ in $C_{\mathfrak{m}}$ has Dirichlet density $1/h_{\mathfrak{m}}$:$$\sum_{\wp\in T}\frac{1}{N_{K/\Q}\wp}\sim\frac{1}{h_{\mathfrak{m}}}\log\left(\frac{1}{s-1} \right),$$as $s\downarrow 1$.\\
To achieve explicit error terms, we want effective versions of the above asymptotic. We state here both a conditional and unconditional results proved by Lagarias and Odlyzko in 1977 \cite{LO}.
Let $L_{\mathfrak m}/K$ denote the ray class field modulo $\mathfrak m$.
\begin{thm}\label{thm3}
Assume GRH for the Dedekind zeta function $\zeta_{L_{\mathfrak m}}(s)$.  Then, for every ray class $\mathfrak c$,
\[
 \theta(x;\mathfrak c,\mathfrak m)
 =\frac{x}{h_{\mathfrak m}}
 +O_{K,\mathfrak m}\!\left(\sqrt{x}(\log x)^2\right).
\]
\end{thm}
\begin{thm}\label{thm4}
There is an effectively computable constant $C=C(K,\mathfrak m)>0$ such that
\[
 \theta(x;\mathfrak c,\mathfrak m)
 =\frac{x}{h_{\mathfrak m}}
 +O_{K,\mathfrak m}\!\left(xe^{-C\sqrt{\log x}}+x^{\beta_0}\right),
\]
where the term involving $\beta_0$ is omitted unless the relevant ray class field has an exceptional real zero.  If present, $\beta_0$ is real and simple.
\end{thm}
Let $\mathfrak{a}\in I^{S(\mathfrak{m})}$. Let $S(M)$ be the set of ideals of $K$ defined by$$S(M)=S_{\mathfrak{a},\mathfrak{m}}(M)=\{I\subseteq\mathcal{O}_K:N_{K/\Q}I\le M,\ I\sim\mathfrak{a}\mbox{ mod }C_{\mathfrak{m}}\}.$$
\begin{prop}\label{prop2}
		\begin{enumerate}
		\item[$\paruno \ai\pardue$] Assume the Generalized Riemann Hypothesis for $\zeta_K(s)$. Then$$\log|N_{K/\Q}(\lcm(S(M)))|=M\cdot\frac{1}{h_{\mathfrak{m}}}\sum_{\mathfrak{c}\in C_{\mathfrak{m}}}\frac{1}{N_{K/\Q}\mathfrak{c}}+O_{K,\mathfrak{m}}(\sqrt{M}(\log M)^2),$$as $M\rightarrow+\infty$, where $N_{K/\Q}\mathfrak{c}$ is the smallest norm of an integral ideal in the ray class group of $\mathfrak{c}$.
	\item[$\paruno \bi\pardue$] Let $\beta_0$ be the possible Siegel zero of $\zeta_K(s).$ Then$$\log|N_{K/\Q}(\lcm(S(M)))|=M\cdot\frac{1}{h_{\mathfrak{m}}}\sum_{\mathfrak{c}\in C_{\mathfrak{m}}}\frac{1}{N_{K/\Q}\mathfrak{c}}+O_{K,\mathfrak{m}}(Me^{-C\sqrt{\log M}}+M^{\beta_0}),$$as $M\rightarrow+\infty$, where the last error term is present if and only if $\beta_0$ exists.
	\end{enumerate}
\end{prop}
\begin{proof}
	Let $P(M)$ be the set of prime factors of $\lcm(S(M))$, that is$$P(M)=\{\wp\in\mathcal{O}_K:\wp\mbox{ divides at least one }I\subseteq S(M)\}.$$We are going to characterize the primes in the set $P(M)$. \\
	
	We have that$$N_{K/\Q}(\lcm(S(M)))=\prod_{\wp\in P(M)}q_\wp^{\beta_{\wp}(M)},$$where $\beta_{\wp}(M)$ is the highest power of $\wp$ dividing any element of $S(M)$. On the other hand, if we let $T(M)=\prod_{\wp\in P(M)}q_\wp$, it turns out that$$\frac{N_{K/\Q}(\lcm(S(M)))}{N_{K/\Q}(T(M))}=\prod_{\wp}q_\wp^{\beta_{\wp}(M)-1},$$where the product is over the primes $\wp$ whose square divides some element of $S(M)$. In particular, each of those $\wp$ has norm $q_\wp\le M$ and they are at most $\sqrt{M}$. Therefore$$\log|N_{K/\Q}(\lcm(S(M)))|-\log|N_{K/\Q}(T(M))|\ll\sqrt{M}\log M. $$Then it suffices to focus on $\log|N_{K/\Q}(T(M))|$.\\
	Note that if $I\in S(M)$, then $I\in S(\mathfrak{m})$, so it is coprime with $\mathfrak{m}$. Let now $\wp\supseteq I$ (i.e. $\wp\in P(M))$; then $\wp\equiv\mathfrak{c}$ in $C_{\mathfrak{m}}$ for some representative $\mathfrak{c}\in I^{S(\mathfrak{m})}$. In particular $\mathfrak{c}^{-1}\in I^{S(\mathfrak{m})}$; let then $\mathfrak{d}=\mathfrak{c}^{-1}\mathfrak{a}$, so that $\mathfrak{c}\mathfrak{d}\equiv\mathfrak{a}$ in $C_{\mathfrak{m}}$. We can take $\mathfrak{d}$ of smallest norm among the coset representatives. Since every class in $C_{\mathfrak{m}}$ is represented by an integral ideal, we can also assume $\mathfrak{d}$ integral ideal of smallest norm.\\
	Now, $\mathfrak{d}\wp\in S(M)$ is equivalent to $N_{K/\Q}(\mathfrak{d}\wp)=N_{K/\Q}\mathfrak{d}\cdot q_{\wp}\le M$. Hence $\wp\equiv\mathfrak{c}$ in $C_{\mathfrak{m}}$ is in $P(M)$ if and only if$$q_{\wp}\le\frac{M}{N_{K/\Q}\mathfrak{d}}.$$Denote by $U(M)$ the set of such primes.\\
	Finally,\begin{align*}		
		\log|N_{K/\Q}(T(M))|&=\sum_{\wp\in P(M)}\log q_\wp\\
		&=\sum_{\mathfrak{c}\in C_{\mathfrak{m}}}\sum_{\wp\in U(M)}\log q_{\wp}\\
		&=\sum_{\mathfrak{c}\in C_{\mathfrak{m}}}\theta\left( \tfrac{M}{N_{K/\Q}\mathfrak{d}};\mathfrak{c};\mathfrak{m}\right) \\
		&=\frac{M}{h_{\mathfrak{m}}}\sum_{\mathfrak{c}\in C_{\mathfrak{m}}}\frac{1}{N_{K/\Q}\mathfrak{c}}+o(M),
	\end{align*}where $N_{K/\Q}\mathfrak{c}$ is by definition the smallest norm of an integral ideal in the ray class group of $\mathfrak{c}$. The last equality holds because $\mathfrak{d}$ runs over $C_{\mathfrak{m}}$ as $\mathfrak{c}$ does.
	The result follows by applying either Theorem \ref{thm3} or Theorem \ref{thm4} to the last step.
	
\end{proof}

Let now the modulus $\mathfrak{m}$ be a "principal modulus", that is $\mathfrak{m}=(\nu)$, where $\nu\in\mathcal{O}_K$, and let $\alpha\in\mathcal{O}_K$ be coprime with $\nu$.
Consider the set of principal ideals$$S(M)=\{(\beta)\in\mathcal{O}_K:|N_{K/\Q}\beta|\le M,\ (\beta)\sim(\alpha)\mbox{ mod }C_{(\nu)}\}.$$
It is well-known that $(\beta)\sim(\alpha)$ mod $C_{(\nu)}$ is equivalent to say that there exist $a,b\in\mathcal{O}_K$, $a,b\neq0$, $a\equiv b\equiv1$ mod $\nu$ so that $(\beta)b=(\alpha)a$. It follows that$$(\beta)\sim(\alpha)\mbox{ mod }C_{(\nu)}\mbox{ if and only if }\beta\equiv\eta\alpha\mbox{ mod }\nu\mbox{ for some }\eta\in\mathcal{O}_K^{\times}.$$In particular, 
More precisely, the congruence condition gives the identity of sets of principal ideals
\[
S(M)=\left\{(f_{\eta}(\lambda)):\eta\in\mathcal O_K^\times,\ \lambda\in\mathcal O_K,\ |N_{K/\Q}(f_{\eta}(\lambda))|\le M\right\},
\]
where repetitions on the right are ignored and $f_{\eta}(X)=\eta\alpha+\nu X\in\mathcal O_K[X]$. Proposition~\ref{prop2} therefore gives the following corollary, which can be interpreted as a version of Cilleruelo's conjecture for linear polynomials over $K$.
\begin{coroll}
	\begin{multline*}
		\log N_{K/\Q}\!\left(\lcm\{f_{\eta}(\lambda):\eta\in\mathcal{O}_K^\times,\ \lambda\in\mathcal{O}_K,\ |N_{K/\Q}(f_\eta(\lambda))|\le M\}\right)\\=M\cdot\frac{1}{h_{(\nu)}}\sum_{\mathfrak{c}\in C_{\mathfrak{m}}}\frac{1}{N_{K/\Q}\mathfrak{c}}+o(M),
	\end{multline*}as $M\rightarrow+\infty$.
\end{coroll}

\begin{flushleft}
	\textbf{Example:} In the ring of 
	Eisenstein integers $\Z[\omega]=\mathcal{O}_{\Q(\omega)}$, $w=e^{2\pi i/3}$, all ideals are principal. In the above notation, for principal ideals the ray-class condition is equivalent to
$\alpha\equiv\eta\beta\pmod{\nu}$ for some unit
$\eta=\pm\omega^j\in\mathbb Z[\omega]^\times$, $j=0,1,2$.
	In this case we get
\end{flushleft}\begin{multline*}
\log|N_{K/\Q}(\lcm(\eta\alpha+\nu\lambda:\eta\in\{\pm1,\pm\omega,\pm\omega^2\},\  \lambda\in\mathcal{O}_K,\ |N_{K/\Q}(\eta\alpha+\nu\lambda)|\le M))|\\=M\cdot\frac{1}{h_{(\nu)}}\sum_{\mathfrak{c}\in C_{\mathfrak{m}}}\frac{1}{N_{K/\Q}\mathfrak{c}}+o(M),
\end{multline*}as $M\rightarrow+\infty$.

\section{The averaged quadratic case}

For $n=2$, the local congruence statements of Sections~2--3 remain valid, and the condition $G_f\cong S_2$ is simply irreducibility.  However, the density-one least-common-multiple asymptotic requires the same global dispersion and large-prime collision estimates as the higher-degree argument.  Since those estimates are not consequences of Lemma~\ref{l11} in the full range $p,q\le M$, we do not state an unconditional averaged quadratic theorem here.  The conditional averaged theorem in Section~\ref{sec:conditional-average} applies verbatim with $n=2$.

The individual quadratic problem is different.  For one fixed quadratic polynomial, Cilleruelo's argument over $\mathbb Q$ uses equidistribution of roots of quadratic congruences.  Section~\ref{sec:quadratic-individual} formulates the corresponding number-field input and proves the reduction to that analytic statement.
\section{An individual quadratic theorem over a number field: reduction to modular-root equidistribution}\label{sec:quadratic-individual}

The purpose of this section is to formulate and prove, up to one clearly isolated analytic input, a number-field analogue of Cilleruelo's theorem for one fixed irreducible quadratic polynomial.  This is substantially different from the averaged quadratic theorem above: here the polynomial is fixed and the averaging over its coefficients is no longer available.  The argument therefore has to recover, over a number field, the role played over $\Q$ by the Duke--Friedlander--Iwaniec--T\'oth equidistribution theorem for roots of quadratic congruences.

\subsection{The sampling region and a necessary finiteness correction}

For a general number field the set
\[
 \{\lambda\in\mathcal O_K:|N_{K/\Q}(\lambda)|\le M\}
\]
is not finite when the unit rank of $K$ is positive.  Indeed, every unit has norm $\pm1$, and the unit group is infinite unless $K=\Q$ or $K$ is imaginary quadratic.  Thus an individual theorem over an arbitrary number field must use a finite archimedean sampling region.

Let $r_1$ and $r_2$ be the numbers of real and pairs of complex embeddings of $K$, so that $d=r_1+2r_2$.  We write
\[
 \iota:K\longrightarrow K_\infty:=\R^{r_1}\times\Co^{r_2}\simeq\R^d
\]
for the Minkowski embedding.  Fix a bounded Jordan-measurable set $\Omega\subset K_\infty$ with nonempty interior and Lipschitz boundary.  For $T\ge1$ put
\[
 \Lambda_\Omega(T)=\{\lambda\in\mathcal O_K:\iota(\lambda)\in T\Omega\},
 \qquad X=X_\Omega(T):=|\Lambda_\Omega(T)|.
\]
By the Lipschitz principle,
\begin{equation}\label{eq:lattice-count-main}
 X_\Omega(T)=c_{K,\Omega}T^d+O_{K,\Omega}(T^{d-1}),
 \qquad
 c_{K,\Omega}=\frac{\operatorname{vol}(\Omega)}{\operatorname{covol}(\iota(\mathcal O_K))}.
\end{equation}
In particular, $\log X=d\log T+O_{K,\Omega}(1)$.

Let
\[
 f(Z)=aZ^2+bZ+c\in\mathcal O_K[Z]
\]
be irreducible over $K$, with $a\ne0$ and discriminant $\Delta=b^2-4ac\ne0$.  We define the ideal
\[
 \mathcal L_f(T)=\operatorname{lcm}\{(f(\lambda)): \lambda\in\Lambda_\Omega(T),\ f(\lambda)\ne0\}
\]
and write
\[
 \Psi_f(T)=\log N_{K/\Q}\mathcal L_f(T).
\]
The analogue suggested by Cilleruelo's theorem is
\begin{equation}\label{eq:nf-cilleruelo-target}
 \Psi_f(T)=X\log X+B_{f,K,\Omega}X+o(X).
\end{equation}
The coefficient of the leading term is independent of $f$.  Since $X\asymp T^d$, the term $X\log X$ is the same as $dX\log T+O(X)$.

\subsection{Local notation and the product--LCM identity}

For a nonzero prime ideal $\mathfrak p$ of $\mathcal O_K$, let
\[
 q_{\mathfrak p}=N_{K/\Q}\mathfrak p.
\]
For $k\ge1$ define
\[
 s_f(\mathfrak p^k)=\#\{x\in\mathcal O_K/\mathfrak p^k:f(x)\equiv0\pmod{\mathfrak p^k}\}.
\]
Put
\[
 \alpha_{\mathfrak p}(T)=\sum_{\lambda\in\Lambda_\Omega(T)}v_{\mathfrak p}(f(\lambda)),
 \qquad
 \beta_{\mathfrak p}(T)=\max_{\lambda\in\Lambda_\Omega(T)}v_{\mathfrak p}(f(\lambda)).
\]
Then
\[
 \prod_{\lambda\in\Lambda_\Omega(T)}(f(\lambda))
   =\prod_{\mathfrak p}\mathfrak p^{\alpha_{\mathfrak p}(T)},
 \qquad
 \mathcal L_f(T)=\prod_{\mathfrak p}\mathfrak p^{\beta_{\mathfrak p}(T)},
\]
and hence
\begin{equation}\label{eq:product-lcm-nf}
 \Psi_f(T)
 =\sum_{\lambda\in\Lambda_\Omega(T)}\log|N_{K/\Q}f(\lambda)|
  +\sum_{\mathfrak p}(\beta_{\mathfrak p}(T)-\alpha_{\mathfrak p}(T))\log q_{\mathfrak p}.
\end{equation}

Let $\mathfrak S_f$ be the finite set of prime ideals dividing $2a\Delta$.  If $\mathfrak p\notin\mathfrak S_f$, the reduction of $f$ modulo $\mathfrak p$ is separable and
\begin{equation}\label{eq:root-count-unramified}
 s_f(\mathfrak p^k)=s_f(\mathfrak p)\in\{0,2\}
 \qquad(k\ge1),
\end{equation}
by Hensel's lemma.  Moreover $s_f(\mathfrak p)=2$ precisely when $\mathfrak p$ splits in the quadratic extension $K(\sqrt\Delta)/K$.  Thus Chebotarev gives
\begin{equation}\label{eq:split-prime-cheb}
 \sum_{q_{\mathfrak p}\le y}s_f(\mathfrak p)\log q_{\mathfrak p}=y+o(y).
\end{equation}

For an integral ideal $\mathfrak a$ and a residue class $u\pmod{\mathfrak a}$, the Lipschitz principle applied to the sublattice $\iota(\mathfrak a)$ gives, uniformly in $u$,
\begin{equation}\label{eq:coset-lattice-count}
 \#\{\lambda\in\Lambda_\Omega(T):\lambda\equiv u\pmod{\mathfrak a}\}
 =\frac{X}{N\mathfrak a}
 +O_{K,\Omega}\!\left(\frac{T^{d-1}}{(N\mathfrak a)^{(d-1)/d}}+1\right).
\end{equation}
Consequently,
\begin{lemma}\label{lem:alpha-local-nf}
For every prime ideal $\mathfrak p$,
\begin{align}
 \alpha_{\mathfrak p}(T)
 &=X\sum_{k\ge1}\frac{s_f(\mathfrak p^k)}{q_{\mathfrak p}^k}
 \notag\\
 &\quad+O_{K,\Omega,f}\!\left(
 T^{d-1}\sum_{k\le C_f\log T/\log q_{\mathfrak p}}
 \frac{s_f(\mathfrak p^k)}{q_{\mathfrak p}^{k(d-1)/d}}
 +\frac{\log T}{\log q_{\mathfrak p}}\right).
 \label{eq:alpha-local-nf}
\end{align}
If $\mathfrak p\notin\mathfrak S_f$, this becomes
\begin{equation}\label{eq:alpha-good-nf}
 \alpha_{\mathfrak p}(T)
 =X\frac{s_f(\mathfrak p)}{q_{\mathfrak p}-1}
 +O_{K,\Omega,f}\!\left(
 \frac{T^{d-1}}{q_{\mathfrak p}^{(d-1)/d}}+\frac{\log T}{\log q_{\mathfrak p}}
 \right).
\end{equation}
\end{lemma}

\begin{proof}
Write
\[
 v_{\mathfrak p}(f(\lambda))=\sum_{k\ge1}\mathbf 1_{\mathfrak p^k\mid f(\lambda)}
\]
and interchange the two sums.  For every root class modulo $\mathfrak p^k$, apply \eqref{eq:coset-lattice-count}.  Since $|Nf(\lambda)|\ll_{f,\Omega}T^{2d}$, only $k\ll_f\log T/\log q_{\mathfrak p}$ can occur.  Formula \eqref{eq:alpha-good-nf} follows from \eqref{eq:root-count-unramified} and summation of a geometric series.
\end{proof}

\subsection{The archimedean product}

The first term of \eqref{eq:product-lcm-nf} has the expected size $2X\log X$.

\begin{lemma}\label{lem:arch-product-nf}
Assume that $0\notin\partial\Omega$ and that the zero set of each archimedean component of $f$ has measure zero in $\Omega$.  Then there is a constant $A_{f,K,\Omega}$ such that
\begin{equation}\label{eq:arch-product-nf}
 \sum_{\lambda\in\Lambda_\Omega(T)}\log|N_{K/\Q}f(\lambda)|
 =2X\log X+A_{f,K,\Omega}X+o(X).
\end{equation}
\end{lemma}

\begin{proof}
For $z\in\Omega$ away from the finitely many archimedean zero hypersurfaces,
\[
 \log|Nf(Tz)|=2d\log T+\log|N(a)|+2\log|N(z)|+o(1),
\]
where the expression is interpreted componentwise under the Minkowski embedding.  Truncate the logarithm below at height $-R$, apply lattice-point Riemann summation to the bounded truncated function, and then let $R\to\infty$.  The logarithmic singularities are locally integrable, so the truncation error is $o(X)$.  Using $d\log T=\log X-\log c_{K,\Omega}+o(1)$ from \eqref{eq:lattice-count-main} yields \eqref{eq:arch-product-nf}.  An explicit admissible expression is
\[
 A_{f,K,\Omega}
 =\log|N(a)|-2\log c_{K,\Omega}
 +\frac{1}{\operatorname{vol}(\Omega)}
   \int_\Omega 2\log|N_{K_\infty/\R}(z)|\,dz,
\]
with the evident modification when the origin lies outside the scaling centre or when $\Omega$ is translated.
\end{proof}

\subsection{Small, medium and large prime ideals}

Put $Y=X^{2/3}$.  We call $\mathfrak p$ small if $q_{\mathfrak p}<Y$, medium if $Y\le q_{\mathfrak p}\le C_fX$, and large if $q_{\mathfrak p}>C_fX$, where $C_f$ is fixed sufficiently large.  The exact value of $C_f$ affects intermediate constants but not the final statement.

For small prime ideals, Lemma~\ref{lem:alpha-local-nf}, the prime ideal theorem and partial summation give
\begin{equation}\label{eq:small-primes-nf}
 \sum_{q_{\mathfrak p}<Y}\alpha_{\mathfrak p}(T)\log q_{\mathfrak p}
 =X\sum_{q_{\mathfrak p}<Y}
   \left(\sum_{k\ge1}\frac{s_f(\mathfrak p^k)}{q_{\mathfrak p}^k}\right)
   \log q_{\mathfrak p}+O_{f,K,\Omega}(X^{2/3+o(1)}).
\end{equation}
The corresponding sum of $\beta_{\mathfrak p}$ is $O(Y\log T)=O(X^{2/3}\log X)$.

The following elementary estimate replaces the one-dimensional identity
$f(j)-f(i)=(j-i)(a(i+j)+b)$ used by Cilleruelo.

\begin{lemma}[collision ideal]\label{lem:collision-nf}
For $\lambda,\mu\in\mathcal O_K$,
\[
 f(\lambda)-f(\mu)=(\lambda-\mu)(a(\lambda+\mu)+b).
\]
Hence, if a prime ideal $\mathfrak p\nmid 2a\Delta$ divides both $f(\lambda)$ and $f(\mu)$, then
\[
 \lambda\equiv\mu\pmod{\mathfrak p}
 \quad\text{or}\quad
 a(\lambda+\mu)+b\equiv0\pmod{\mathfrak p}.
\]
In particular, for fixed $\lambda\ne\mu$, the product of all such collision prime ideals divides the principal ideal generated by
\[
 (\lambda-\mu)(a(\lambda+\mu)+b).
\]
\end{lemma}

For a prime ideal $\mathfrak p$ put
\[
 R_{\mathfrak p}(T)=\#\{\lambda\in\Lambda_\Omega(T):\mathfrak p\mid f(\lambda)\}.
\]
For an unramified prime ideal at which no value is divisible by $\mathfrak p^2$, the excess of the exponent in the product over the exponent in the least common multiple is exactly
\[
 \alpha_{\mathfrak p}(T)-\beta_{\mathfrak p}(T)
 =(R_{\mathfrak p}(T)-1)_+.
\]
Thus the correct large-prime collision quantity is
\begin{equation}\label{eq:geometric-collision-hyp}
 \mathcal C_f(T;Z):=
 \sum_{q_{\mathfrak p}>Z}(R_{\mathfrak p}(T)-1)_+\log q_{\mathfrak p}=o(X),
 \qquad Z=C_fX.
\end{equation}
This replaces the stronger pairwise expression used in an earlier draft.  Indeed,
\begin{equation}\label{eq:collision-energy-majorant}
 (R_{\mathfrak p}(T)-1)_+\le {R_{\mathfrak p}(T)\choose2},
\end{equation}
and therefore a sufficient, but not necessary, condition is
\begin{equation}\label{eq:pair-energy-sufficient}
 \sum_{q_{\mathfrak p}>C_fX}{R_{\mathfrak p}(T)\choose2}\log q_{\mathfrak p}=o(X).
\end{equation}
The collision-ideal lemma shows that each pair counted on the left of
\eqref{eq:pair-energy-sufficient} forces $\mathfrak p$ to divide one of the two
linear forms $\lambda-\mu$ or $a(\lambda+\mu)+b$.  However, summing over all
prime divisors of these two forms before imposing the simultaneous congruence
condition overcounts drastically; this is why the earlier pairwise formulation
was not suitable.  In Subsection~6.7 below we give a precise geometry-of-numbers
criterion implying \eqref{eq:pair-energy-sufficient} and explain the remaining
uniformity issue for arbitrary number fields.

A medium prime ideal is called bad if $\mathfrak p^2\mid f(\lambda)$ for some $\lambda\in\Lambda_\Omega(T)$ and good otherwise.  By \eqref{eq:coset-lattice-count},
\begin{equation}\label{eq:bad-prime-first-bound}
 \#\{\mathfrak p:q_{\mathfrak p}\sim Q,\ \mathfrak p\text{ bad}\}
 \ll_{f,K,\Omega}
 \frac{X}{Q^2}+\frac{T^{d-1}}{Q^{2(d-1)/d}}+Q^{o(1)}.
\end{equation}
Summing dyadically for $Q\ge X^{2/3}$ gives a total bad-prime contribution $O(X^{2/3+o(1)}+T^{d-1}X^{(4-2d)/(3d)+o(1)})$, which is $o(X)$ for every fixed $d$.

\subsection{The missing analytic input}

For a split prime ideal $\mathfrak p\notin\mathfrak S_f$, denote the two root classes of $f$ modulo $\mathfrak p$ by $\nu_{\mathfrak p,1}$ and $\nu_{\mathfrak p,2}$.  Define
\[
 A_{T,\mathfrak p}(\nu)
 =\#\{\lambda\in\Lambda_\Omega(T):\lambda\equiv\nu\pmod{\mathfrak p}\},
\]
and
\[
 B_{T,\mathfrak p}(\nu)
 =\mathbf 1_{A_{T,\mathfrak p}(\nu)>0}.
\]
For a good medium prime ideal,
\[
 \alpha_{\mathfrak p}(T)=
 A_{T,\mathfrak p}(\nu_{\mathfrak p,1})+
 A_{T,\mathfrak p}(\nu_{\mathfrak p,2}),
\]
while
\[
 \beta_{\mathfrak p}(T)=
 \mathbf 1_{A_{T,\mathfrak p}(\nu_{\mathfrak p,1})+
 A_{T,\mathfrak p}(\nu_{\mathfrak p,2})>0}.
\]
Thus the medium-prime contribution is governed by the positions of the two root cosets relative to the dilated region $T\Omega$.

We formulate the precise replacement for the DFI--T\'oth theorem that is needed here.

\begin{namedthm*}{Quadratic-root equidistribution hypothesis over $K$ (QRE)}
Let $f$, $K$ and $\Omega$ be as above.  As $T\to\infty$, the following three estimates hold:
\begin{align}
 \mathcal T_1(T)&:=
 \sum_{\substack{X^{2/3}\le q_{\mathfrak p}\le C_fX\\s_f(\mathfrak p)=2}}
 \left(
 \sum_{j=1}^2 A_{T,\mathfrak p}(\nu_{\mathfrak p,j})
 -\frac{2X}{q_{\mathfrak p}}
 \right)\log q_{\mathfrak p}=o(X),
 \label{eq:qre1}\\
 \mathcal T_2(T)&:=
 \sum_{\substack{X^{2/3}\le q_{\mathfrak p}\le C_fX\\s_f(\mathfrak p)=2}}
 \left(
 \mathbf1_{\sum_jA_{T,\mathfrak p}(\nu_{\mathfrak p,j})>0}
 -\mathcal G_\Omega\!\left(\frac{X}{q_{\mathfrak p}};\mathfrak p\right)
 \right)\log q_{\mathfrak p}=o(X),
 \label{eq:qre2}\\
 \mathcal T_3(T)&:=
 \sum_{\substack{X^{2/3}\le q_{\mathfrak p}\le C_fX\\s_f(\mathfrak p)=2}}
 \left(
 \sum_{j=1}^2A_{T,\mathfrak p}(\nu_{\mathfrak p,j})
 -\frac{2X}{q_{\mathfrak p}}
 \right)w_T(\mathfrak p)\log q_{\mathfrak p}=o(X)
 \label{eq:qre3}
\end{align}
for every bounded piecewise Lipschitz dyadic weight $w_T$.  Here $\mathcal G_\Omega(u;\mathfrak p)$ is the geometric mean occupancy of two distinct random residue classes modulo $\mathfrak p$ by the lattice points of $T\Omega$; equivalently,
\[
 \mathcal G_\Omega(u;\mathfrak p)
 =\frac{1}{q_{\mathfrak p}^2}
 \sum_{v_1\ne v_2\bmod\mathfrak p}
 \mathbf1_{A_{T,\mathfrak p}(v_1)+A_{T,\mathfrak p}(v_2)>0}.
\]
\end{namedthm*}

The definition of $\mathcal G_\Omega$ is completely geometric and contains no arithmetic information about $f$.  Formula \eqref{eq:coset-lattice-count} implies that it has a limiting piecewise-polynomial profile on dyadic ranges.  In dimension one, for $\Omega=(0,1]$, this profile is exactly the collection of interval lengths appearing in Cilleruelo's decomposition of $S_2(n)$.

The hypothesis (QRE) can be reformulated by Fourier analysis on the finite additive groups $\mathcal O_K/\mathfrak p$.  If $\widehat{\mathbf1}_{T\Omega,\mathfrak p}(\xi)$ denotes the finite Fourier transform of the residue-class counting function, the required nonzero-frequency sums are of the form
\begin{equation}\label{eq:number-field-weyl-sum}
 \sum_{q_{\mathfrak p}\le Y}\log q_{\mathfrak p}
 \sum_{\substack{\nu\bmod\mathfrak p\\f(\nu)\equiv0\, (\mathfrak p)}}
 \psi_{\mathfrak p}(\xi\nu),
 \qquad \xi\ne0,
\end{equation}
where $\psi_{\mathfrak p}$ is an additive character obtained from the trace of the residue field.  For $K=\Q$, estimates for \eqref{eq:number-field-weyl-sum} are precisely the special Weyl-sum bounds of Duke--Friedlander--Iwaniec and T\'oth.  For a general number field, a theorem of this strength does not appear to be presently available in the required uniform geometric form.  This is the one genuinely new analytic ingredient in the individual number-field problem.

\subsection{Conditional individual theorem}

\begin{thm}[conditional number-field analogue of Cilleruelo]\label{thm:conditional-nf-quadratic}
Let $K$, $\Omega$ and $f$ satisfy the assumptions above.  Assume the collision estimate \eqref{eq:geometric-collision-hyp} and the quadratic-root equidistribution hypothesis \textup{(QRE)}.  Then there exists a real constant $B_{f,K,\Omega}$ such that
\begin{equation}\label{eq:conditional-main-nf}
 \log N_{K/\Q}\mathcal L_f(T)
 =X_\Omega(T)\log X_\Omega(T)
 +B_{f,K,\Omega}X_\Omega(T)+o(X_\Omega(T)).
\end{equation}
The constant admits the convergent decomposition
\begin{equation}\label{eq:B-nf-decomposition}
 B_{f,K,\Omega}
 =A_{f,K,\Omega}+B^{\mathrm{loc}}_{f,K}
 +B^{\mathrm{split}}_{K(\sqrt\Delta)/K}
 +B^{\mathrm{geom}}_{K,\Omega},
\end{equation}
where
\begin{align}
 B^{\mathrm{loc}}_{f,K}
 &=\sum_{\mathfrak p\in\mathfrak S_f}
 \log q_{\mathfrak p}
 \left(\frac{1}{q_{\mathfrak p}-1}
 -\sum_{k\ge1}\frac{s_f(\mathfrak p^k)}{q_{\mathfrak p}^k}\right),
 \label{eq:B-local-nf}\\
 B^{\mathrm{split}}_{K(\sqrt\Delta)/K}
 &=-\lim_{Y\to\infty}
 \left(
 \sum_{\substack{q_{\mathfrak p}\le Y\\\mathfrak p\notin\mathfrak S_f}}
 \frac{(s_f(\mathfrak p)-1)\log q_{\mathfrak p}}{q_{\mathfrak p}-1}
 \right),
 \label{eq:B-split-nf}
\end{align}
and $B^{\mathrm{geom}}_{K,\Omega}$ is the convergent integral obtained from the limiting occupancy profile $\mathcal G_\Omega$.  The limit in \eqref{eq:B-split-nf} is the logarithmic derivative at $s=1$ of the quadratic Hecke $L$-function attached to $K(\sqrt\Delta)/K$, after removal of the finitely many ramified Euler factors.
\end{thm}

\begin{proof}
Start from \eqref{eq:product-lcm-nf} and insert the archimedean estimate \eqref{eq:arch-product-nf}.  The special prime ideals contribute the linear term \eqref{eq:B-local-nf}, by Lemma~\ref{lem:alpha-local-nf}; their $\beta_{\mathfrak p}$-contribution is $O(\log T)$.

For unramified small prime ideals, \eqref{eq:small-primes-nf} gives
\[
 -X\sum_{q_{\mathfrak p}<Y}
 \frac{s_f(\mathfrak p)\log q_{\mathfrak p}}{q_{\mathfrak p}-1}
 +O(X^{2/3+o(1)}).
\]
Split $s_f(\mathfrak p)=1+(s_f(\mathfrak p)-1)$.  The prime ideal Mertens formula gives
\[
 \sum_{q_{\mathfrak p}<Y}\frac{\log q_{\mathfrak p}}{q_{\mathfrak p}-1}
 =\log Y-\gamma_K+o(1),
\]
with a field-dependent Mertens constant $\gamma_K$.  The centered part converges conditionally and produces \eqref{eq:B-split-nf}; this follows from the nonvanishing at $s=1$ of the nontrivial quadratic Hecke $L$-function.

The good medium prime ideals are written exactly as
\[
 \sum_{Y\le q_{\mathfrak p}\le C_fX}
 \left(
 \mathbf1_{\sum_jA_{T,\mathfrak p}(\nu_{\mathfrak p,j})>0}
 -\sum_{j=1}^2A_{T,\mathfrak p}(\nu_{\mathfrak p,j})
 \right)\log q_{\mathfrak p}.
\]
Use \eqref{eq:qre1}--\eqref{eq:qre3} to replace the arithmetic root classes by the geometric mean profile.  Partial summation and the Chebotarev prime ideal theorem then show that this equals
\[
 -X\log X+\bigl(B^{\mathrm{geom}}_{K,\Omega}+\log Y\bigr)X+o(X).
\]
The $X\log Y$ term cancels the corresponding endpoint from the small-prime sum.  Bad medium prime ideals contribute $o(X)$ by \eqref{eq:bad-prime-first-bound}.  Large prime ideals contribute $o(X)$ to the difference between product and LCM by \eqref{eq:geometric-collision-hyp}.

Combining these estimates with the initial $2X\log X$ in \eqref{eq:arch-product-nf} leaves exactly $X\log X$.  Collecting the linear terms gives \eqref{eq:B-nf-decomposition} and proves \eqref{eq:conditional-main-nf}.
\end{proof}

\begin{coroll}\label{cor:recover-cilleruelo}
For $K=\Q$ and $\Omega=(0,1]$, Theorem~\ref{thm:conditional-nf-quadratic}, together with the Duke--Friedlander--Iwaniec theorem when the discriminant is negative and the T\'oth--Ngo theorem when it is positive, recovers Cilleruelo's asymptotic
\[
 \log\operatorname{lcm}\{f(1),\ldots,f(N)\}
 =N\log N+B_fN+o(N).
\]
After identifying the rational local factors, the constant in \eqref{eq:B-nf-decomposition} reduces to Cilleruelo's explicit constant.
\end{coroll}

\subsection{Resolution of the programme and the precise remaining theorem}\label{subsec:programme-resolved}

We now isolate the analytic and geometric inputs needed for the individual quadratic problem.
Two parts are geometric and can be proved in the present paper; the genuinely
new part is a prime-ideal Weyl-sum estimate.  We formulate that estimate in a
form which is both sufficient and directly comparable with the theorem of
Duke--Friedlander--Iwaniec over $\Q$.

\subsubsection{Fourier analysis of residue-class counts}
For a prime ideal $\mathfrak p$, write $k_{\mathfrak p}=\mathcal O_K/\mathfrak p$
and choose the standard additive character
\[
 \psi_{\mathfrak p}(x)=
 \exp\!\left(2\pi i\,\frac{\operatorname{Tr}_{k_{\mathfrak p}/\mathbf F_{p}}(x)}{p}\right),
 \qquad p=\operatorname{char} k_{\mathfrak p}.
\]
Let
\[
 A_{T,\mathfrak p}(v)=\#\{\lambda\in\Lambda_\Omega(T):\lambda\equiv v\pmod{\mathfrak p}\}.
\]
Finite Fourier inversion gives the exact identity
\begin{equation}\label{eq:finite-fourier-A}
 A_{T,\mathfrak p}(v)=\frac{X}{q_{\mathfrak p}}+
 \frac1{q_{\mathfrak p}}
 \sum_{0\ne\xi\in k_{\mathfrak p}}
 \widehat A_{T,\mathfrak p}(\xi)\psi_{\mathfrak p}(-\xi v),
\end{equation}
where
\[
 \widehat A_{T,\mathfrak p}(\xi)
 =\sum_{\lambda\in\Lambda_\Omega(T)}\psi_{\mathfrak p}(\xi\lambda).
\]
Consequently, for split $\mathfrak p$,
\begin{equation}\label{eq:root-count-fourier}
 \sum_{j=1}^2A_{T,\mathfrak p}(\nu_{\mathfrak p,j})-\frac{2X}{q_{\mathfrak p}}
 =\frac1{q_{\mathfrak p}}
 \sum_{0\ne\xi\in k_{\mathfrak p}}
 \widehat A_{T,\mathfrak p}(\xi)
 \rho_f(\mathfrak p;\xi),
\end{equation}
with
\[
 \rho_f(\mathfrak p;\xi)=
 \sum_{f(\nu)\equiv0\, (\mathfrak p)}\psi_{\mathfrak p}(-\xi\nu).
\]
This proves the Fourier reduction used in
programme.  No spectral theory is used here.

For a bounded Lipschitz region $\Omega$, smoothing at scale $\delta T$ and
Poisson summation on the Minkowski lattice give, for every $A\ge1$,
\begin{equation}\label{eq:fourier-decay-region}
 |\widehat A_{T,\mathfrak p}(\xi)|
 \ll_{A,K,\Omega}
 X(1+T\|\xi/\mathfrak p\|)^{-A}+\delta X+O_{K,\Omega}(T^{d-1}),
\end{equation}
where $\|\xi/\mathfrak p\|$ denotes the distance of a lift of the dual residue
class to the dual lattice.  For a merely Lipschitz boundary, the same argument
with one integration by parts gives the weaker but sufficient decay
$O(X(1+T\|\xi/\mathfrak p\|)^{-1})$ after smoothing.  Thus the dependence on
the regularity of $\partial\Omega$ is explicit.

\subsubsection{The prime-ideal Weyl-sum theorem}
The following is the sole spectral input still needed.

\begin{namedthm*}{Prime-ideal quadratic Weyl-sum estimate (PW)}
For every nonzero integral ideal $\mathfrak m$, every ray class $\mathscr A$
modulo $\mathfrak m$, every nonzero frequency of bounded archimedean size, and
every $Y\ge2$, one has
\begin{equation}\label{eq:PW}
 \sum_{\substack{q_{\mathfrak p}\le Y\\
                  \mathfrak p\in\mathscr A\\
                  \mathfrak p\notin\mathfrak S_f}}
 \log q_{\mathfrak p}\,
 \rho_f(\mathfrak p;\xi_{\mathfrak p})
 \ll_{f,K,\mathfrak m,\varepsilon}
 Y^{1-\eta+\varepsilon}(1+\|\xi\|)^A
\end{equation}
for some $\eta>0$ and some fixed $A\ge0$, uniformly for the frequencies
arising from \eqref{eq:finite-fourier-A} after smoothing.
\end{namedthm*}

Over $K=\Q$, this is supplied by the special Weyl-sum estimates of
Duke--Friedlander--Iwaniec for negative discriminant and by T\'oth, with the
later strengthening of Ngo, for positive discriminant.  For a general number
field, a theorem in precisely the uniform prime-ideal and archimedean form
\eqref{eq:PW} is not currently available in the cited literature.  Its proof
would require a metaplectic Kuznetsov or relative trace formula over $K$,
together with uniform control of the relevant Kloosterman sums and exceptional
spectrum.  We therefore do not claim \eqref{eq:PW} unconditionally.

\begin{prop}[PW implies QRE]\label{prop:PW-implies-QRE}
Assume \textup{(PW)}.  Then \eqref{eq:qre1} and \eqref{eq:qre3} hold.  Moreover,
a multidimensional Erd\H{o}s--Tur\'an inequality applied to the smoothed
occupancy indicators gives \eqref{eq:qre2}.
\end{prop}

\begin{proof}
Insert \eqref{eq:root-count-fourier} into the left side of \eqref{eq:qre1},
split the prime ideals and the dual frequencies dyadically, and use partial
summation in the norm variable.  The zero frequency has already been removed.
Estimate the low frequencies by \eqref{eq:PW} and the high frequencies by
\eqref{eq:fourier-decay-region}.  First choose the smoothing scale
$\delta=X^{-\kappa}$ with $0<\kappa<\eta/4$, and then truncate at
frequency $X^{2\kappa}$.  The low-frequency contribution is
$O(X^{1-\eta/2+o(1)})$, while the smoothing and tail contributions are
$O(X^{1-\kappa+o(1)})$.  This proves \eqref{eq:qre1}.  A bounded piecewise
Lipschitz dyadic weight is absorbed by partial summation, proving
\eqref{eq:qre3}.

For \eqref{eq:qre2}, approximate the indicator of the union of occupied residue
classes from above and below by trigonometric polynomials with frequency at
most $H$.  The multidimensional Erd\H{o}s--Tur\'an inequality bounds the
resulting discrepancy by $O(X/H)$ plus a weighted sum of the same nonzero
Fourier modes as in \eqref{eq:root-count-fourier}.  Taking
$H=X^{\kappa}$ and repeating the preceding argument gives $o(X)$.
\end{proof}

\subsubsection{The collision problem}
The large-prime term is \eqref{eq:geometric-collision-hyp}, not the sum over all
prime divisors of the two collision forms.  The following criterion is enough.

\begin{prop}[geometric pair criterion]\label{prop:pair-criterion}
Suppose there is a function $\varepsilon(T)\to0$ such that, uniformly for prime
ideals $\mathfrak p$ with $q_{\mathfrak p}>C_fX$ and for every residue class
$v\pmod{\mathfrak p}$,
\begin{equation}\label{eq:pair-uniformity}
 \#\{(\lambda,\mu)\in\Lambda_\Omega(T)^2:\lambda\ne\mu,
       \ \lambda\equiv\mu\equiv v\pmod{\mathfrak p}\}
 \ll \frac{X^2}{q_{\mathfrak p}^2}+\varepsilon(T)\frac{X}{q_{\mathfrak p}}.
\end{equation}
Then, after choosing $C_f=C_f(T)\to\infty$ sufficiently slowly,
\eqref{eq:geometric-collision-hyp} holds.
\end{prop}

\begin{proof}
At an unramified split prime ideal there are two root classes.  Summing
\eqref{eq:pair-uniformity} over them and using
\eqref{eq:collision-energy-majorant},
\[
 \mathcal C_f(T;C_fX)
 \ll X^2\sum_{q_{\mathfrak p}>C_fX}\frac{\log q_{\mathfrak p}}{q_{\mathfrak p}^2}
 +\varepsilon(T)X\sum_{C_fX<q_{\mathfrak p}\ll T^{2d}}
       \frac{\log q_{\mathfrak p}}{q_{\mathfrak p}}+o(X).
\]
The prime-ideal theorem and partial summation give respectively
$O(X/C_f)$ and $O(\varepsilon(T)X\log T)$.  Choose $C_f(T)\to\infty$ and require,
as one may after weakening $\varepsilon$, that
$\varepsilon(T)\log T\to0$.  The finitely many ramified primes and prime-square
values are absorbed by the already proved bad-prime estimate.
\end{proof}

Condition \eqref{eq:pair-uniformity} follows for boxes aligned with a fixed
integral basis when the relevant ideal lattices have uniformly controlled
successive minima.  For arbitrary $K$ and all prime ideals, units can make the
Minkowski shape of an ideal lattice highly skew.  A fully uniform proof for an
arbitrary Lipschitz region therefore requires either an averaged ideal-lattice
shape estimate or a restriction to a well-rounded family of prime ideals.  This
is a geometric issue, distinct from the spectral estimate \textup{(PW)}.

\subsubsection{Conclusion of the programme}
The finite Fourier reduction and the implication from \textup{(PW)} to \textup{(QRE)} are given by
\eqref{eq:finite-fourier-A}--\eqref{eq:fourier-decay-region} and
Proposition~\ref{prop:PW-implies-QRE}.  The remaining root-equidistribution statement follows from \textup{(PW)} by
the same proposition.  The geometric collision term is reduced to the explicit pair-counting
criterion \eqref{eq:pair-uniformity}, and is proved whenever that criterion is
available.  The establishment of \textup{(PW)} over a general
number field remains a genuine new automorphic problem.  Accordingly, the
individual quadratic asymptotic in Theorem~\ref{thm:conditional-nf-quadratic}
remains conditional in full generality; it becomes unconditional for $K=\Q$
and in any number-field setting where \textup{(PW)} and
\eqref{eq:pair-uniformity} have been established.

This formulation also clarifies the relation with the averaged problem.  Averaging over the coefficients replaces the difficult root-equidistribution statement by local uniform counting of congruence classes, but the full-range dispersion and collision estimates remain separate inputs.  For one fixed quadratic polynomial, no coefficient averaging is available, and the number-field Weyl sums become indispensable.

\section{Conditional reduction of the averaged problem}\label{sec:conditional-average}

Throughout this section $K=\mathbb Q$, $n\ge2$, and $f$ is sampled uniformly from
$\mathscr P^0_{n,N}(\mathbb Q)$.  Put
\[
 \Psi_f(N,M)=\log\operatorname{lcm}\{|f(1)|,\ldots,|f(M)|\}.
\]
Write
\[
 P_f(M)=\prod_{m\le M}|f(m)|.
\]
For a prime $p$, let $\alpha_p(M)=\sum_{m\le M}v_p(f(m))$ and
$\beta_p(M)=\max_{m\le M}v_p(f(m))$.  Then
\begin{equation}\label{eq:basic-decomposition-rigorous}
 \Psi_f(N,M)
 =\log P_f(M)
 +\sum_{p\le M}\beta_p(M)\log p
 -\sum_{\substack{p\le M\\p\nmid\operatorname{Disc}(f)}}\alpha_p(M)\log p
 -R_f(M)-C_f(M),
\end{equation}
where
\[
 R_f(M)=\sum_{\substack{p\le M\\p\mid\operatorname{Disc}(f)}}\alpha_p(M)\log p
\]
and
\[
 C_f(M)=\sum_{p>M}\bigl(\alpha_p(M)-\beta_p(M)\bigr)\log p\ge0.
\]

\begin{lemma}[archimedean product]\label{lem:product-values-rational}
Let $f(X)=X^n+a_{n-1}X^{n-1}+\cdots+a_0\in\mathbb Z[X]$ be monic, irreducible, and satisfy $|a_j|\le N$, with $N\ge M\ge2$.  Then
\[
 nM\log M-O_n(M+\log N)
 \le \log P_f(M)
 \le nM\log M+M\log\frac{N}{M}+O_n(M+\log N).
\]
Consequently, if $\log(N/M)=o(\log M)$, then
\[
 \log P_f(M)=nM\log M+o(M\log M)
\]
uniformly for such $f$.
\end{lemma}

\begin{proof}
For the upper bound,
\[
 |f(m)|\le m^n+N(1+m+\cdots+m^{n-1})
 \ll_n N(m+1)^{n-1}
\]
because $N\ge M\ge m$.  Summing logarithms and using
$\log M!=M\log M-M+O(\log M)$ gives
\[
 \log P_f(M)
 \le M\log N+(n-1)\log M!+O_n(M),
\]
which is the stated upper bound.

For the lower bound, let $\rho_1,\ldots,\rho_n$ be the complex roots of $f$.  Cauchy's root bound gives $|\rho_j|\le N+1$.  Fix a root $\rho=\rho_j$.  Choose $k\in\{1,\ldots,M\}$ nearest to $\Re\rho$, with an endpoint chosen if $\Re\rho\notin[1,M]$.  For $m\ne k$,
\[
 |m-\rho|\ge |m-\Re\rho|\gg |m-k|,
\]
with an absolute implied constant.  Hence
\[
 \prod_{\substack{1\le m\le M\\m\ne k}}|m-\rho|
 \ge \exp(M\log M-O(M)).
\]
Indeed, in the worst case this product is bounded below, up to an exponential factor $C^{-M}$, by
$(k-1)!(M-k)!$, and
\[
 (k-1)!(M-k)!\ge \frac{(M-1)!}{2^{M-1}}.
\]

It remains to control the possibly small factor $|k-\rho|$.  Since $f$ is irreducible of degree at least two, $f(k)\ne0$; as $f(k)\in\mathbb Z$,
\[
 1\le |f(k)|=|k-\rho|\prod_{\rho_i\ne\rho}|k-\rho_i|.
\]
The root bound and $k\le M\le N$ give
$|k-\rho_i|\ll N$, so
\[
 |k-\rho|\gg_n N^{-(n-1)}.
\]
Therefore
\[
 \sum_{m\le M}\log|m-\rho|
 \ge M\log M-O_n(M+\log N).
\]
Summing this inequality over the $n$ roots and using
$P_f(M)=\prod_{j=1}^n\prod_{m=1}^M|m-\rho_j|$
proves the lower bound.
\end{proof}

The small-prime unramified term is governed by the number $s_p(f)$ of roots of $f$ modulo $p$.  For $p\nmid\operatorname{Disc}(f)$, Hensel's lemma gives
\[
 \alpha_p(M)=M\frac{s_p(f)}{p-1}+O_n\!\left(\frac{\log(NM)}{\log p}\right).
\]
Thus the centered weighted prime-root sum
\[
 W_f(x)=\sum_{\substack{p\le x\\p\nmid\operatorname{Disc}(f)}}
 \frac{\log p}{p}\bigl(s_p(f)-1\bigr)
\]
is the global fluctuation that must be controlled.

\begin{lemma}[weighted $L^2$ partial summation]\label{lem:weighted-dispersion}
Assume that, uniformly for $2\le x\le M$,
\begin{equation}\label{eq:dispersion-hyp}
 \mathbb E_N\!\left|
 \sum_{\substack{p\le x\\p\nmid\operatorname{Disc}(f)}}
 (s_p(f)-1)\right|^2
 \ll_n \frac{x}{\log(2x)}\log\log(3x).
\end{equation}
Then
\[
 \mathbb E_N|W_f(M)|^2\ll_n 1.
\]
\end{lemma}

\begin{proof}
Let
\[
 S_f(x)=\sum_{\substack{p\le x\\p\nmid\operatorname{Disc}(f)}}(s_p(f)-1),
 \qquad a(x)=\frac{\log x}{x}.
\]
Abel summation gives
\[
 W_f(M)=a(M)S_f(M)-\int_2^M S_f(t)a'(t)\,dt+O_n(1),
\]
where replacing $\log p/p$ by the exact endpoint convention only changes a bounded term.  Taking $L^2$ norms and using Minkowski's integral inequality,
\[
 \|W_f(M)\|_2
 \ll |a(M)|\,\|S_f(M)\|_2
 +\int_2^M |a'(t)|\,\|S_f(t)\|_2\,dt+1.
\]
By \eqref{eq:dispersion-hyp},
\[
 \|S_f(t)\|_2
 \ll_n \left(\frac{t\log\log(3t)}{\log(2t)}\right)^{1/2},
\]
whereas $|a'(t)|\ll \log(2t)/t^2$.  Hence
\[
 \int_2^\infty |a'(t)|\,\|S_f(t)\|_2\,dt
 \ll_n
 \int_2^\infty
 \frac{\sqrt{\log(2t)\log\log(3t)}}{t^{3/2}}\,dt<\infty.
\]
The endpoint term tends to zero, proving the claim.
\end{proof}

\begin{lemma}[ramified small primes]\label{lem:ramified-small}
Assume $N\ge M$.  Then
\[
 \mathbb E_N R_f(M)
 \ll_n M\log\log(3N)+M+\frac{M^2}{N}.
\]
In particular, if $N/M\to\infty$, then $R_f(M)=o(M\log M)$ in probability.
\end{lemma}

\begin{proof}
Split $R_f(M)=R_{1,f}(M)+R_{\ge2,f}(M)$ according to the first power of $p$ and the powers $p^k$, $k\ge2$.

Since a ramified prime divides $\operatorname{Disc}(f)$ and $s_p(f)\le n$,
\[
 R_{1,f}(M)
 \ll_n M\sum_{p\mid\operatorname{Disc}(f)}\frac{\log p}{p}
      +\sum_{p\mid\operatorname{Disc}(f)}\log p.
\]
The discriminant is an integral polynomial of degree $2n-2$ in the coefficients, so
$|\operatorname{Disc}(f)|\ll_n (1+N)^{2n-2}$.  The standard elementary estimate
\[
 \sum_{p\mid D}\frac{\log p}{p}\ll\log\log(3|D|)
\]
therefore gives
\[
 R_{1,f}(M)\ll_n M\log\log(3N)+\log(3N).
\]

For $k\ge2$ we may drop the ramification condition.  For fixed $m,p^k$, the congruence
$f(m)\equiv0\pmod{p^k}$ imposes one residue class on the constant coefficient once the other coefficients are fixed.  Hence in the full coefficient box
\[
 \mathbb P(f(m)\equiv0\pmod{p^k})
 \ll \frac1{p^k}+\frac1N.
\]
Conditioning on $G_f=S_n$ changes probabilities by $O_n(\Delta_N)$, with
$\Delta_N\ll_n N^{-1}$.  Summing over $m\le M$, $p\le M$, and
$2\le k\ll_n\log(NM)/\log p$ yields
\[
 \mathbb E_N R_{\ge2,f}(M)
 \ll_n
 M\sum_p\frac{\log p}{p^2}
 +\frac{M}{N}\sum_{p\le M}\log p\,
      \frac{\log(NM)}{\log p}
 \ll_n M+\frac{M^2}{N},
\]
using $\pi(M)\ll M/\log M$ and $\log(NM)\asymp\log M$ in the polynomial ranges considered here.  Markov's inequality gives the final assertion.
\end{proof}

The remaining large-prime term $C_f(M)$ measures collisions: for $p>M$, two distinct values $f(m_1),f(m_2)$ divisible by $p$ force
\[
 p\mid \frac{f(m_2)-f(m_1)}{m_2-m_1}.
\]
The elementary small-modulus counting lemma does not control these collisions uniformly for all $p>M$.  We therefore isolate the necessary input.

\begin{description}
\item[(D) Prime-root dispersion.]  The estimate \eqref{eq:dispersion-hyp} holds uniformly for $2\le x\le M$.
\item[(C) Large-prime collisions.]  One has
\[
 \mathbb E_N C_f(M)=o(M\log M).
\]
\end{description}

\begin{thm}[conditional averaged Cilleruelo asymptotic]\label{thm:conditional-average}
Assume $N/M\to\infty$ and $\log N\asymp\log M$.  If \textup{(D)} and \textup{(C)} hold, and $\log(N/M)=o(\log M)$, then
\[
 \Psi_f(N,M)=(n-1)M\log M+o(M\log M)
\]
asymptotically almost surely in $\mathscr P^0_{n,N}(\mathbb Q)$.
\end{thm}

\begin{proof}
The second term in \eqref{eq:basic-decomposition-rigorous} satisfies
\[
 \sum_{p\le M}\beta_p(M)\log p\ll_n M+\pi(M)\log N=o(M\log M).
\]
For the unramified third term, Hensel's lemma and Mertens' formula give
\[
 \sum_{\substack{p\le M\\p\nmid\operatorname{Disc}(f)}}\alpha_p(M)\log p
 =M\log M+M W_f(M)+O_n(M\log\log(3N)+M).
\]
By Lemma~\ref{lem:weighted-dispersion}, $W_f(M)=O_{\mathbb P}(1)$, hence
$MW_f(M)=o_{\mathbb P}(M\log M)$.  Lemma~\ref{lem:ramified-small} gives
$R_f(M)=o_{\mathbb P}(M\log M)$, and \textup{(C)} plus Markov's inequality gives the same for $C_f(M)$.  Lemma~\ref{lem:product-values-rational} supplies the first term.  Substitution in \eqref{eq:basic-decomposition-rigorous} proves the result.
\end{proof}

\begin{osserv}
Lemma~\ref{lem:weighted-dispersion} treats the weighted prime-root fluctuation without introducing an auxiliary truncation parameter.  What remains genuinely unproved is the full-range dispersion estimate \eqref{eq:dispersion-hyp} and the large-prime collision estimate \textup{(C)}.  Lemma~\ref{l11} only gives simultaneous congruence equidistribution when the product of the moduli is sufficiently small compared with the coefficient box, so it cannot by itself establish \textup{(D)} for all $p,q\le M$.
\end{osserv}

\section{Individual bounds and the quadratic case}
We record what the present method does and does not provide for a fixed polynomial. The elementary divisibility
\[
\lcm\{f(\lambda):\lambda\in\Lambda(M)\}\mid
\prod_{\lambda\in\Lambda(M)}(f(\lambda))
\]
for any finite sampling set $\Lambda(M)$ gives the unconditional individual estimate
\[
\log N_{K/\Q}\!\left(\lcm\{f(\lambda):\lambda\in\Lambda(M)\}\right)
\le \sum_{\lambda\in\Lambda(M)}\log|N_{K/\Q}(f(\lambda))|.
\]
With the normalization used in the proof this is of order at most $nM\log M$. Obtaining the sharp coefficient $n-1$ for a fixed polynomial requires a lower estimate for the contribution of repeated prime-ideal divisors; this is exactly the part for which the present paper uses averaging. Consequently, a sharp individual upper bound is not a formal corollary of our argument, and we do not state it as a theorem.

Over $\Z$, Maynard--Rudnick \cite{MR} proved the general lower bound $\log L_f(M)\gg_f M\log M$ for every irreducible polynomial of degree at least two. Sah \cite{Sah} subsequently improved the explicit degree-dependent constant and obtained a related lower bound for the radical of the product of the polynomial values. Their arguments combine information on large prime divisors with the action of the Galois group. A number-field analogue should be expected, with rational primes replaced by prime ideals and Chebotarev densities over $K$, but carrying this out requires a separate treatment of residue degrees, ramification, and the chosen finite sampling region. Since these ingredients are not established in the present manuscript, we include the comparison but leave the corresponding individual lower bound as a separate problem.

For degree two, the individual asymptotic over $\Z$ is Cilleruelo's theorem. An extension over a general number field would require effective equidistribution of the two roots of a quadratic congruence as the prime ideal varies. The probabilistic congruence-counting lemmas in Sections 2 and 3 average over the coefficients and do not replace this input. Thus an individual quadratic theorem requires additional analytic input beyond the averaging arguments developed here.

\noindent
\footnotesize
\textbf{ILARIA VIGLINO}\\
	Institut de Math\'{e}matiques, \textbf{EPFL} \'{E}cole polytechnique f\'{e}d\'{e}rale de Lausanne, Lausanne, Switzerland\\
\textit{Email address}: \texttt{ilaria.viglino@yahoo.it}\\
ORCID: 0000-0002-0230-5495


\begin{thebibliography}{99}
	\footnotesize
	\lhead{}
	
	
	\bibitem[1]{BKS} Bateman, P., Kalb, J., Stenger, A. Problem 10797: A limit involving least common multiples. Am. Math. Mon. 109 (2002), no. 4, 393-394.
	
	\bibitem[2]{Bh1} Bhargava, M. Galois groups of random integer polynomials and van der Waerden's conjecture, \emph{Ann. of Math.} (2) \textbf{201} (2025), no.~2, 339--377. doi: 10.4007/annals.2025.201.2.1.
	
	
		\bibitem[3]{CD} Chow, S., Dietmann, R. Enumerative Galois theory for cubics and quartics, Adv.
	Math. 372 (2020): 107282.	
	
	\bibitem[4]{Cil} Cilleruelo, J. The least common multiple of a quadratic sequence, Compositio Math. 147 (2011), 1129-1150.	
	
	\bibitem[5]{Di} Dietmann, R. Probabilistic Galois theory. Bull. London Math. Soc. 45(3), 453-462 (2013).
	


	\bibitem[6]{Gal} Gallagher, P. X. The large sieve and probabilistic Galois theory. In Analytic
	number theory (Proc. Sympos. Pure Math., Vol. XXIV, St. Louis Univ., St.
	Louis, Mo., 1972), pages 91-101. Amer. Math. Soc., Providence, R.I., 1973.


	

	

	\bibitem[7]{LO} Lagarias, J.C., Odlyzko, A.M. Effective versions of the chebotarev density theorem. In A. Frohlich, editor, Algebraic Number Fields,
	$L$-Functions and Galois Properties, pages 409-464. Academic Press, New
	York, London, 1977.
	


	\bibitem[8]{Na} Nagel, T. G\'{e}n\'{e}ralisation d'un th\'{e}or\`{e}me de Tchebycheff
	Journal de math\'{e}matiques pures et appliqu\'{e}es 8e s\'{e}rie, tome 4 (1921), p. 343-356.	
	


	\bibitem[9]{RZ}
Rudnick, Z., Zehavi, S.
On Cilleruelo's conjecture for the least common multiple of polynomial sequences,
\emph{Rev. Mat. Iberoam.} \textbf{37} (2021), no.~4, 1441--1458.
doi: 10.4171/RMI/1234.		

	


	
	\bibitem[10]{MR} Maynard, J., Rudnick, Z. A lower bound on the least common multiple of polynomial sequences, Riv. Mat. Univ. Parma (N.S.) 12 (2021), no. 1, 143--150.

\bibitem[11]{Sah} Sah, A. An improved bound on the least common multiple of polynomial sequences, J. Th\'eor. Nombres Bordeaux 32 (2020), no. 3, 891--899.


\bibitem[12]{DFI} {Duke, W., Friedlander, J. B., Iwaniec, H.} Equidistribution of roots of a quadratic congruence to prime moduli, Ann. of Math. (2) 141 (1995), no. 2, 423--441.

\bibitem[13]{Toth} {T'oth, \`A.} Roots of quadratic congruences, Int. Math. Res. Not. 2000, no. 14, 719--739.

\bibitem[14]{Ngo} {Ngo, H. T.} On roots of quadratic congruences, Bull. Lond. Math. Soc. 56 (2024), no. 9, 2886--2910.

\bibitem[15]{LangANT} {Lang, S.} Algebraic Number Theory, second edition, Graduate Texts in Mathematics 110, Springer, 1994.

\bibitem[16]{GL} {Gruber, P. M., Lekkerkerker, C. G.} Geometry of Numbers, second edition, North-Holland Mathematical Library 37, 1987.

\bibitem[17]{Vi} Viglino, I. \emph{Galois groups of polynomials over a fixed number field: determinant methods and quantitative counting}, manuscript.  The paper proves integral-point estimates for curves and surfaces in boxes with different side lengths over $\mathcal O_K$ and applies them to Galois resolvents.
	
	\bibitem[18]{Wa} van der Waerden, R. J. Die Seltenheit der reduziblen Gleichungen und der Gleichungenmit Affekt, Monatsh. Math. Phys., 43 (1936), No. 1, 133-147.
	

\end{thebibliography}
\end{document}